\documentclass[11pt,reqno]{amsart}

\newtheorem{proposition}{Proposition}[section]

\newtheorem{theorem}[proposition]{Theorem}

\theoremstyle{definition}
\newtheorem{definition}[proposition]{Definition}
\newtheorem{example}[proposition]{Example}
\newtheorem{examples}[proposition]{Examples}
\newtheorem{remark}[proposition]{Remark}

\newcommand{\thlabel}[1]{\label{th:#1}}
\newcommand{\thref}[1]{Theorem~\ref{th:#1}}
\newcommand{\selabel}[1]{\label{se:#1}}
\newcommand{\seref}[1]{Section~\ref{se:#1}}

\newcommand{\exlabel}[1]{\label{ex:#1}}
\newcommand{\exref}[1]{Example~\ref{ex:#1}}
\newcommand{\delabel}[1]{\label{de:#1}}
\newcommand{\deref}[1]{Definition~\ref{de:#1}}
\newcommand{\eqlabel}[1]{\label{eq:#1}}
\newcommand{\equref}[1]{(\ref{eq:#1})}

\def\RR{{\mathbb R}}

\newcommand{\Cc}{\mathcal{C}}

\def\*C{{}^*\hspace*{-1pt}{\Cc}}
\def\text#1{{\rm {\rm #1}}}

\input xy
\xyoption {all} \CompileMatrices

\usepackage{amssymb}
\usepackage{color,amssymb,graphicx,amscd,amsmath,dsfont,stmaryrd,xcolor,comment}
\usepackage[pagebackref,colorlinks,urlcolor=blue,linkcolor=blue,citecolor=blue]{hyperref}

\begin{document}

\title[The factorization problem for Jordan algebras. Applications]
{The factorization problem for Jordan algebras. Applications}

\author{A. L. Agore}
\address{Simion Stoilow Institute of Mathematics of the Romanian Academy, P.O. Box 1-764, 014700 Bucharest, Romania and Vrije Universiteit Brussel, Pleinlaan 2, B-1050 Brussels, Belgium}
\email{ana.agore@vub.be and ana.agore@gmail.com}

\author{G. Militaru}
\address{Faculty of Mathematics and Computer Science, University of Bucharest, Str.
Academiei 14, RO-010014 Bucharest 1, Romania and Simion Stoilow Institute of Mathematics of the Romanian Academy, P.O. Box 1-764, 014700 Bucharest, Romania}
\email{gigel.militaru@fmi.unibuc.ro and gigel.militaru@gmail.com}
\subjclass[2010]{16T10, 16T05, 16S40}

\thanks{This work was supported by a grant of the Ministry of Research, Innovation and Digitization, CNCS/CCCDI--UEFISCDI, project number
PN-III-P4-ID-PCE-2020-0458, within PNCDI III}

\subjclass[2020]{17C10, 17C50, 17C55} \keywords{Matched pair and bicrossed product of Jordan algebras; The factorization problem.}

\begin{abstract}
We investigate the factorization problem as well as the classifying complements problem in the setting of Jordan algebras. Matched pairs of Jordan algebras and the corresponding bicrossed products are introduced. It is shown that any Jordan algebra which factorizes through two given Jordan algebras is isomorphic to a bicrossed product associated to a certain matched pair between the same two Jordan algebras. Furthermore, a new type of deformation of a Jordan algebra is proposed as the main step towards solving the classifying complements problem.
\end{abstract}

\maketitle

\section*{Introduction}
The factorization problem is an old and notoriously difficult problem which stems in group theory. A weaker version of the factorization problem was first considered in \cite{Ma, Or} where groups $G$ which admit two subgroups $A$ and $B$ such that $G = AB$, are studied under the name of permutable groups. If we assume, in addition, that $A$ and $B$ have trivial intersection then we say that $G$ \emph{factorizes} through $A$ and $B$. The factorization problem for groups asks for the description and classification of all groups which factorize through two given groups. As explained in \cite{abm1, acim}, despite its elementary statement, the factorization problem is far from being an easy question. An important turning point for the factorization problem was the introduction of the matched pairs of groups \cite{Tak} which consist of two groups that act on each other in a compatible way. To each matched pair of groups one can associate a so-called \emph{bicrossed product}, i.e. a group structure on the direct product of the underlying sets constructed from the two actions. This leads to a new and more computational approach to the factorization problem. Consequently, the factorization problem comes down to finding all matched pairs between two given groups and classifying the corresponding bicrossed products. Therefore, the same strategy relying on matched pairs and bicrossed products was used to approach the factorization problem for various mathematical objects such as: (co)algebras \cite{Tom, CIMZ, cap}, Lie
algebras and Lie groups \cite{Y, majid3}, Leibniz algebras \cite{am-2013b}, Hopf algebras \cite{majid}, fusion categories \cite{Gk} and so on. Furthermore, this new approach has also the advantage of opening the way to new classification methods as evidenced in \cite{a, abm1, an, K} for Hopf algebras. Since their introduction by P. Jordan in $1933$, Jordan algebras have appeared in various different fields of mathematics and mathematical physics such us the theory of superstrings, supersymmetry, projective geometry, Lie algebras and algebraic groups, representation theory or functional analysis \cite{Mc}. In the setting of Jordan algebras, the factorization problem can be stated as follows:

\textbf{Factorization problem.} \textit{Let $A$ and $V$ be two
given Jordan algebras. Describe and classify up to an isomorphism
all Jordan algebras $E$ that factorize through $A$ and $V$, i.e.
$E$ contains $A$ and $V$ as Jordan subalgebras such that $E = A +
V$ and $A \cap V = \{0\}$.}

Another closely-related problem which we will consider here was first
introduced in \cite{am-2013} at the level of Lie/Hopf algebras as
a converse of the factorization problem. Corresponding theories have been developed for associative algebras \cite{a2} or groups \cite{am-2015} where an important connection with the problem of classifying all groups of a given finite order was highlighted. For Jordan algebras it comes down to the following:

\textbf{Classifying complements problem.} \textit{Let $A \subseteq
E$ be a Jordan subalgebra of a Jordan algebra $E$. If a complement
of $A$ in $E$ exists (i.e. a Jordan subalgebra $V\subseteq E$ such
that $E = A + V$ and $A \cap V = \{0\}$), describe and classify up
to an isomorphism all others complements of $A$ in $E$.}

The paper is devoted to the two problems listed above and is organized as follows: in \seref{prel} we set the notation and recall some
basic definitions and properties of Jordan algebras. \seref{cazurispeciale} focuses on the factorization problem for Jordan algebras. \deref{mpmL} introduces matched pairs of Jordan algebras and \thref{bicross} connects them to the bicrossed product construction. These are the
Jordan algebra counterparts of similar constructions performed for groups \cite{Tak} or Lie algebras \cite{LW, majid}. \thref{bicromlver} proves that the
bicrossed product of two Jordan algebras is the object responsible
for the factorization problem and is the Jordan algebra version of
\cite[Theorem 3.9]{LW}. As the examples in this section show, most low dimensional Jordan algebras can be written as bicrossed products of certain Jordan subalgebras. This suggests that achieving the classification of bicrossed products can lead to a successful strategy for classifying Jordan algebras. To this end, a general classification result for bicrossed products is presented in \thref{morfismebicross}.

The last section of the paper deals with the classifying complements problem for Jordan algebras following the general strategy we developed previously in \cite{a2, am-2013, am-2015} for groups and associative/Lie/Hopf algebras. Consequently, a new type of \emph{deformation} of a Jordan algebra arising from a matched pair is introduced in \thref{deforJ}. More precisely, these deformations involve a certain linear map $r \colon B \to A$ associated to the canonical matched pair $(A,\, B,\, \triangleleft,\, \triangleright)$, called a deformation map. It is shown that all $A$-complements of a Jordan algebra $E$ are obtained from a given complement $B$ by this new kind of deformation. More precisely, for each deformation map $r \colon B \to A$ we can define a new $A$-complement denoted by $B_{r}$ and, conversely, for any $A$-complement $B'$ we can find a deformation map of the canonical matched pair $(A,\, B,\, \triangleleft,\, \triangleright)$ such that $B' \cong B_{r}$. Relevant examples are presented. In particular, it is proved that $4$ out of the $6$ isomorphism types of real Jordan algebras of dimension $2$ (\cite[Theorem 1]{anc}) appear as deformations of a given $2$-dimensional Jordan algebra. This further substantiate the importance of these objects for the classification problem of Jordan algebras.

\section{Preliminaries}\selabel{prel}
All vector spaces, (bi)linear maps, tensor products are over a 
field $k$ of characteristic $\neq 2$. A \emph{Jordan algebra} is a vector space $A$ together with a
bilinear map $\cdot : A \times A \to A$, called multiplication, which is commutative and satisfies the Jordan identity, i.e. for any $a$, $b\in A$ we have:
\begin{equation*}
a\cdot b = b\cdot a, \qquad (a^2 \cdot b) \cdot a = a^2 \cdot (b
\cdot a)
\end{equation*}
A Jordan algebra $A$ with trivial multiplication, i.e. $a \cdot b := 0$ for all $a$, $b\in A$, will be called \emph{abelian} and a vector space $V$ endowed with the abelian multiplication will be denoted by $V_0$. Throughout, when there is no fear of confusion, we will denote the multiplication of a Jordan algebra just by juxtaposition. 
As in the case of Lie algebras, any associative algebra induces a Jordan algebra. More precisely, given an associative algebra $A$ one can define a Jordan algebra structure on the underlying vector space by $x \cdot y :=
2^{-1} (xy + yx)$, for all $x$, $y\in A$. The automorphism group of the Jordan algebra $A$ will be denoted by ${\rm Aut}_{\rm
J}(A)$. We recall from \cite{am2022} the following concept: 

\begin{definition}\delabel{moduleJ}
A \emph{right action} of a Jordan algebra $A$ on a vector space
$M$ is a bilinear map $\triangleleft : M \times A \to M$ such that
for any $a\in A$ and $x \in M$ we have:
\begin{eqnarray} \eqlabel{actiunedr}
(x \triangleleft a^2) \triangleleft a = (x \triangleleft a)
\triangleleft a^2
\end{eqnarray}
Similarly, a \emph{left action} of $A$ on $M$ is a bilinear map
$\triangleright : A \times M \to M$ such that for any $a\in A$ and
$x \in M$ we have:
\begin{eqnarray} \eqlabel{actiunest}
a \triangleright  (a^2 \triangleright x) = a^2 \triangleright (a
\triangleright x)
\end{eqnarray}
\end{definition}

The canonical maps $\triangleright : A \times A \to A$ and
$\triangleright : A \times A^* \to A^*$ given for any $a$, $b\in
A$ and $a^* \in A^*$ by:
\begin{equation} \eqlabel{calduramare}
a \triangleright b := a \cdot b, \qquad (a \triangleright a^*) (b)
:= a^* (a\cdot b)
\end{equation}
are left actions of $A$ on $A$ and $A^* = {\rm Hom}_k
\, (A, \, k) $, respectively.

\begin{definition}\delabel{moduleJ}
A \emph{Jordan bimodule} \cite{CDDV} over a Jordan algebra $A$ is a vector space
$M$ together with two linear maps:
\begin{eqnarray*}
&& M \otimes A \to M, \,\,\, x \otimes a \mapsto xa\\
&& A \otimes M \to M, \,\,\, a \otimes x \mapsto ax
\end{eqnarray*}
subject to the following compatibilities:
\begin{eqnarray}
&& xa = ax\\\
&& a(a^{2}x) = a^{2}(xa)\eqlabel{bim2}\\
&& (a^{2}b)x - a^{2}(bx) = 2 \bigl[(ab)(ax) - a(b(ax))\bigl]\eqlabel{bim3}
\end{eqnarray}
\end{definition}

For all unexplained notions pertaining to Jordan algebra theory we refer the reader to \cite{jacb, Ko, Mc}. 

\section{Matched pairs and bicrossed products of Jordan algebras}\selabel{cazurispeciale}

In this section we introduce matched pairs of Jordan algebras and the corresponding bicrossed product in order to approach the factorization problem. Part of our results regarding matched pairs will be derived as a special case from the more general theory developed in \cite[Section 2]{am2022} in connection to the extending structures problem.

\begin{definition}\delabel{mpmL}
Let $A$ and $V$ be two Jordan algebras. Then the quadruple $(A, V, \triangleleft, \, \triangleright)$ is
called a \emph{matched pair of Jordan algebras} if $\triangleleft
: V\times A \to V$ is a right action of the Jordan algebra $A$ on
$V$, $\triangleright : V \times A \to A$ is a left action of the
Jordan algebra $V$ on $A$ and the following compatibility conditions hold for all $a$, $b \in A$
and $x$, $y \in V$:
\begin{eqnarray*}
&& \hspace*{-8mm} {\rm (MP1)}\,\, a \, (x \triangleright a^2) + (x
\triangleleft a^2 )\triangleright a = a^2 \, (x \triangleright a) + (x \triangleleft a) \triangleright a^{2};\\
&& \hspace*{-8mm}{\rm (MP2)}\,\, x \triangleleft (x^2 \triangleright a) + (x^2
\triangleleft a)  \, x = x^2 \triangleleft (x
\triangleright a) + x^2  \, (x \triangleleft a);\\
&& \hspace*{-8mm} {\rm (MP3)}\,\, 2\bigl[\bigl((x \triangleleft a) \triangleleft b\bigl) \triangleleft \,a + x \triangleleft \bigl((x \triangleright a) b\bigl) + x \triangleleft \bigl((x \triangleleft a) \triangleright b\bigl) + \bigl((x \triangleleft a) \triangleleft b\bigl)x \bigl] + (x^{2} \triangleleft b) \triangleleft a +\\
&& \,\,\,\,\,\, x \triangleleft (ba^{2}) = 2\bigl[(x \triangleleft a) \triangleleft ba + (x \triangleleft a) \triangleleft (x \triangleright b) + (x \triangleleft b) \triangleleft (x \triangleright a) + (x \triangleleft a)(x \triangleleft b)\bigl] +\\
&& \,\,\,\,\,\, x^{2} \triangleleft ba + (x \triangleleft b) \triangleleft a^{2};\\
&& \hspace*{-8mm} {\rm (MP4)}\,\, 2\bigl[y \triangleright (x \triangleright a) + x \triangleright \bigl(y \triangleright (x \triangleright  a)\bigl) + \bigl(y \triangleleft (x \triangleright a)\bigl) \triangleright\, a + \bigl((x \triangleleft a) y\bigl) \triangleright a\bigl] + (x^{2}y) \triangleright a +\\
&& \,\,\,\,\,\, x \triangleright (y \triangleright a^{2}) = 2\bigl[(y \triangleright a)(x \triangleright a) + xy \triangleright (x \triangleright a) + (y \triangleleft a) \triangleright (x \triangleright a) + (x \triangleleft a) \triangleright (y \triangleright a)\bigl] +\\
&& \,\,\,\,\,\, x^{2} \triangleright (y \triangleright a) + xy \triangleright a^{2};\\
&& \hspace*{-8mm} {\rm (MP5)}\,\, 2\bigl[\bigl(y \triangleleft (x \triangleright a)\bigl) \triangleleft \,a + \bigl((x \triangleleft a)y\bigl) \triangleleft a + x \triangleleft \bigl(y \triangleright (x \triangleright a)\bigl) + \bigl(y \triangleleft (x \triangleright a)\bigl)x + \bigl((x \triangleleft a)y\bigl)x\bigl] +\\
&& \,\,\,\,\,\, x^{2}y \triangleleft a + x \triangleleft (y \triangleright a^{2}) + (y \triangleleft a^{2}) x = 2\bigl[(x \triangleleft a) \triangleleft (y \triangleright a) + (y \triangleleft a) \triangleleft (x \triangleright a) + \\
&& \,\,\,\,\,\, xy \triangleleft (x \triangleright a) + (x \triangleleft a)xy + (x \triangleleft a)(y \triangleleft  a)\bigl] + x^{2} \triangleleft (y \triangleright a) + xy \triangleleft a^{2} + x^{2}(y \triangleleft a); \\
&& \hspace*{-8mm} {\rm (MP6)}\,\, 2\bigl[\bigl((x \triangleright a)b\bigl)a + \bigl((x \triangleleft a) \triangleright b\bigl)a + \bigl((x \triangleleft a) \triangleleft b\bigl) \triangleright a + x \triangleright \bigl((x \triangleright a)b\bigl) + x \triangleright \bigl((x \triangleleft a) \triangleright b\bigl)\bigl] +\\
&& \,\,\,\,\,\, (x^{2} \triangleright b)a + (x^{2} \triangleleft b)  \triangleright a + x  \triangleright (a^{2}b) = 2 \bigl[(x \triangleright a)(ab) + (x \triangleleft a) \triangleright (ba) + (x  \triangleright a) (x  \triangleright b) +\\
&& \,\,\,\,\,\, (x  \triangleleft b) \triangleright (x \triangleright a) + (x \triangleleft a) \triangleright (x \triangleright b)\bigl] + x^{2} \triangleright (ba) + (x \triangleleft b) \triangleright a^{2} +a^{2}(x \triangleright b).
\end{eqnarray*}
\end{definition}

We denote by $ A \bowtie V$ the
vector space $A \, \times V$ together with the bilinear map $
\circ : (A \times V) \times (A \times V) \to A \times V$ defined
by:
\begin{equation}\eqlabel{bracbicross}
(a, x) \circ (b, y) := \bigl( a b + x \triangleright b
+ y\triangleright a , \,\,  x\triangleleft b + y\triangleleft a +
x y  \bigl)
\end{equation}

The next result provides a motivation for the matched pair axioms as introduced in \deref{mpmL} and is the Jordan algebra version of \cite[Theorem 4.1]{majid}:

\begin{theorem}\thlabel{bicross}
Let $A$ and $V$ be two Jordan algebras and $\triangleleft
\colon V\times A \to V$, $\triangleright \colon V \times A \to A$ two bilinear maps. The multiplication defined in \equref{bracbicross} is a Jordan algebra structure on $A \times V$ if and only if $(A, V, \triangleleft, \, \triangleright)$ is a matched pair of Jordan algebras. In this case, the Jordan algebra $A
\bowtie V$ will be called the \emph{bicrossed product} associated to the matched pair $(A, V, \triangleleft, \, \triangleright)$.
\end{theorem}

\begin{proof} Note that the proof can be performed in a direct manner by a rather long and laborious computation. Instead, we will apply the results from \cite[Section 2]{am2022}. Indeed, by \cite[Definition 2.2]{am2022}, a Jordan extending datum of $A$ through $V$ with $f := 0$ comes down precisely to the definition of a matched pair while the multiplication of the corresponding unified product reduces to the one in \equref{bracbicross}. The conclusion now follows from \cite[Theorem 2.4]{am2022} and \cite[Example  2.6]{am2022}.
\end{proof}

Our first examples of a bicrossed product are generic ones, namely the \emph{semidirect product of Jordan algebras} \cite{jacb} and the \emph{null split extension of a Jordan algebra by a bimodule} \cite[Section 2.2]{KOS}.

\begin{example} \exlabel{clifordJ}
$1.$ Let $A$, $V$ be two Jordan algebras and $\triangleright : V
\times A \to A$ a left action such that the following
compatibilities hold for all $a$, $b \in A$ and $x$, $y \in V$:
\begin{eqnarray*} \eqlabel{Lsemidirectpr}
&& \hspace*{-8mm}{\rm (L1)}\,\,\,\,\,\, a \, ( x \triangleright a^2) = a^2 \, (x \triangleright a)\\
&&  \hspace*{-8mm}{\rm (L2)}\,\,\,\,\,\, 2\bigl[y \triangleright (x \triangleright a) - (y \triangleright a)(x \triangleright a) + x \triangleright \bigl(y \triangleright (x \triangleright  a)\bigl) -xy \triangleright (x \triangleright a) \bigl]\\
&&\quad =  x^{2} \triangleright (y \triangleright a) -  (x^{2}y) \triangleright a +  (xy) \triangleright a^{2} -  x \triangleright (y \triangleright a^{2})\\
&& \hspace*{-8mm}{\rm (L3)}\,\,\,\,\,\,  2\bigl[\bigl((x \triangleright a)b\bigl)a - (x \triangleright a)(ab) + x \triangleright \bigl((x \triangleright a)b\bigl) -  (x  \triangleright a) (x  \triangleright b)\bigl]\\
&&\quad = x^{2} \triangleright (ba) - (x^{2} \triangleright b)a + a^{2}(x \triangleright b) - x \triangleright (a^{2}b)
\end{eqnarray*}
Then the direct product $A \times V$ together with the multiplication given as follows $a$, $b \in A$ and $x$, $y \in V$:
\begin{equation}\eqlabel{semiL}
(a, x) \circ (b, y) := \bigl( a b + x \triangleright b
+ y\triangleright a , \,\, x y  \bigl)
\end{equation}
is a Jordan algebra called the left semidirect product and will be denoted by $A \ltimes V$.

It can be easily seen that the left semidirect product can be recovered from \deref{mpmL} by considering a matched pair $(A, V, \triangleleft, \, \triangleright)$ with $\triangleleft := 0$. Then (MP2), (MP3) and (MP5) are trivially fulfilled while (MP1), (MP4) and (MP6) reduce to the compatibilities (L1), (L2) and (L3) respectively.

Similarly, the right semidirect product also appears as a special case of the bicrossed product. To this end, consider $A$, $V$ to be two Jordan algebras and $\triangleleft : V \times A \to V$ a right action such that the following compatibilities hold for all $a$, $b \in A$ and $x$, $y \in V$:
\begin{eqnarray*} \eqlabel{Rsemidirectpr}
&& \hspace*{-8mm}{\rm (R1)}\,\,\,\,\,\, x\, (x^2
\triangleleft a) = x^2  \, (x \triangleleft a)\\
&&  \hspace*{-8mm}{\rm (R2)}\,\,\,\,\,\,
2\bigl[\bigl((x \triangleleft a) \triangleleft b\bigl) \triangleleft \,a - (x \triangleleft a) \triangleleft ba + \bigl((x \triangleleft a) \triangleleft b\bigl)x - (x \triangleleft a)(x \triangleleft b)\bigl] \\
&&\quad =  x^{2} \triangleleft ba - (x^{2} \triangleleft b) \triangleleft a +  (x \triangleleft b) \triangleleft a^{2} - x \triangleleft (ba^{2})\\
&& \hspace*{-8mm}{\rm (R3)}\,\,\,\,\,\,  2\bigl[\bigl((x \triangleleft a)y\bigl) \triangleleft\, a - (x \triangleleft a)(y \triangleleft  a) + \bigl((x \triangleleft a)y\bigl)x - (x \triangleleft a)xy\bigl]\\
&&\quad = xy \triangleleft a^{2} - (y \triangleleft a^{2}) x + x^{2}(y \triangleleft a) - x^{2}y \triangleleft a
\end{eqnarray*}
Then the direct product $A \times V$ together with the multiplication given as follows $a$, $b \in A$ and $x$, $y \in V$:
\begin{equation}\eqlabel{semiR}
(a, x) \circ (b, y) := \bigl( a b, \,\, x\triangleleft b + y\triangleleft a + x y  \bigl)
\end{equation}
is a Jordan algebra called the right semidirect product and will be denoted by $A \rtimes V$.

The right semidirect product can be recovered from \deref{mpmL} by considering a matched pair $(A, V, \triangleleft, \, \triangleright)$ with $\triangleright := 0$. Then (MP1), (MP4) and (MP6) are trivially fulfilled while (MP2), (MP3) and (MP5) reduce to the compatibilities (R1), (R2) and (R3) respectively.

$2.$ Let $A$ be a Jordan algebra and $M$ a Jordan bimodule over $A$. We will see $M$ as an abelian Jordan algebra, i.e. $xy = 0$ for all $x$, $y \in M$. It is straightforward to see that $\triangleleft \colon M \otimes A \to M$ defined by $x \triangleleft a = xa$ is a right action of $A$ on $M$ satisfying $(R1)-(R3)$, where we denote by juxtaposition the bimodule structure on $M$. Indeed, to start with, $\triangleleft$ is a right action by virtue of \equref{bim2}. Furthermore, $(R1)$ and $(R3)$ hold trivially since we assumed $M$ to be an abelian Jordan algebra and by the same argument $(R2)$ follows from \equref{bim3}. The multiplication on the corresponding right semidirect product $A \ltimes M$ is given as follows for all $a$, $b \in A$ and $x$, $y \in M$ by:
\begin{eqnarray*}\eqlabel{semiR}
(a, x) \circ (b, y) := \bigl( a b, \,\, xb + ya  \bigl)
\end{eqnarray*}
which is precisely the null split extension of a Jordan algebra by a bimodule.
\end{example}

\begin{remark}
A straightforward computation shows that $V \cong \{0\} \times V$ is an ideal
of $A \rtimes V$, $A \cong A  \times \{0\}$ is a subalgebra of $A
\rtimes V$ and the canonical inclusion $i_A : A \to A \rtimes V$,
$i_A (a) = (a, 0)$ is a split monomorphism of Jordan algebras: its
retraction $\pi_A : A \rtimes V \to A$, $\pi_A (a, x) := a$, for
all $a\in A$ and $x\in V$ is a Jordan algebra map.

Conversely, the right semidirect product of Jordan algebras describes
split monomorphisms in the category: more precisely, if $i: A \to E$ is a split
monomorphism of Jordan algebras with splitting map $p: E \to A$, then $E \cong A \rtimes V$, where $V = {\rm ker} p$ is obviously a Jordan subalgebra of $E$.
Indeed, $\psi \colon A \rtimes V \to E$ defined by $\psi(a,\,x) = a+x$ for all $(a,\,x) \in A \rtimes V$ is a Jordan algebra isomorphism, where $A \rtimes V$  is the right semidirect product corresponding to the right action $\triangleleft : V \times A \to V$ given by $x \triangleleft a = xa$. To this end, for all $a$, $b \in A$ and $x$, $y \in V$ we have:
\begin{eqnarray*}
\psi(a,\,x) \psi(b,\, y) &=& (a+x)(b+y) = ab+ay+xb+xy = \psi(ab,\, xb+ya+xy)\\
&=& \psi \bigl((a,\,x) \circ (b,\,y)\bigl)
\end{eqnarray*}
\end{remark}

In what follows, when describing a matched pair of Jordan algebras, we only indicate the non-zero values of the two actions.

\begin{examples}\exlabel{mpbpex}

$1.$ \cite[Section 4.1]{Ka} lists the $7$ isomorphism classes of non-associative unitary Jordan algebras of dimension $4$. It can be easily seen that all those Jordan algebras can be written as bicrossed products of certain subalgebras. We include here one example; $J_{5}$ is the $4$-dimensional real Jordan algebra with multiplication table defined as follows:
\begin{center}
\begin{tabular} {l | c  c  c  c  c}
$J_{5}$ & $a$ & $b$ & $u$ & $v$\\
\hline $a$ & $a$ & $0$ & $\frac{1}{2}u$ & $v$\\
$b$ & $0$ & $b$ & $\frac{1}{2}u$ & $0$ \\
$u$ & $\frac{1}{2}u$ & $\frac{1}{2}u$ & $0$ & $0$\\
$v$ & $v$ & $0$ & $0$ & $0$ \\
\end{tabular}
\end{center}
$J_{5}$ is a right semidirect product between the $2$-dimensional Jordan algebras $A \,=\, <a,\, b ~|~ a^{2} = a,\, b^{2} = b, \, ab =0>$ and the abelian $2$-dimensional Jordan algebra $V$ generated by $u$ and $v$ corresponding to the right action $\triangleleft \colon V \times A \to V$ defined as follows:
$$
v \triangleleft a = v, \qquad u \triangleleft a = u \triangleleft b = \frac{1}{2} u.
$$

$2.$ Similarly, among the $19$ isomorphism classes of non-associative unitary Jordan algebras of dimension $5$ listed in \cite[Section 4.2]{Ka} there is only one which can not be written as a bicrossed product, namely $J_{3}$. We include below two examples: the Jordan algebras denoted in \cite{Ka} by $J_{7}$ and $J_{17}$. $J_{7}$ is the $5$-dimensional real Jordan algebra with multiplication table defined as follows:

\begin{center}
\begin{tabular} {l | c  c  c  c  c  c}
$J_{7}$ & $a$ & $b$ & $c$ & $u$ & $v$\\
\hline $a$ & $a$ & $0$ & $\frac{1}{2} c$ & $\frac{1}{2} u$ & $\frac{1}{2} v$\\
$b$ & $0$ & $b$ & $\frac{1}{2} c$ & $\frac{1}{2} u$ & $\frac{1}{2} v$\\
$c$ & $\frac{1}{2} c$ & $\frac{1}{2} c$ & $0$ & $\frac{1}{2} (a+b)$ & $0$\\
$u$ & $\frac{1}{2} u$ & $\frac{1}{2} u$ & $\frac{1}{2} (a+b)$ & $0$ & $0$\\
$v$ & $\frac{1}{2} v$ & $\frac{1}{2} v$ & $0$ & $0$ & $0$\\
\end{tabular}
\end{center}
$J_{7}$ is a bicrossed product between the $3$-dimensional Jordan algebra $A\, =\, <a,\, b,\,c ~|~ a^{2} = a,\, b^{2} = b,\, c^{2} = 0,\, ab=0,\, ac = \frac{1}{2}c,\, bc= \frac{1}{2}c>$ and the abelian $2$-dimensional Jordan algebra $V$ generated by $u$ and $v$. The actions $(\triangleleft, \, \triangleright)$ of the associated matched pair are defined as follows:
\begin{eqnarray*}
&& u \triangleleft a = u \triangleleft b = \frac{1}{2} u,\,\,\, v \triangleleft a = v \triangleleft b =\frac{1}{2} v,\,\,\, u \triangleright c = \frac{1}{2}(a+b)
\end{eqnarray*}

$J_{17}$ is the $5$-dimensional real Jordan algebra with the following multiplication table:
\begin{center}
\begin{tabular} {l | c  c  c  c  c  c}
$J_{17}$ & $a$ & $b$ & $c$ & $u$ & $v$\\
\hline $a$ & $a$ & $b$ & $\frac{1}{2} c$ & $0$ & $\frac{1}{2} v$\\
$b$ & $b$ & $0$ & $0$ & $0$ & $0$\\
$c$ & $\frac{1}{2} c$ & $0$ & $0$ & $\frac{1}{2} c$ & $0$\\
$u$ & $0$ & $0$ & $\frac{1}{2} c$ & $u$ & $\frac{1}{2} v$\\
$v$ & $\frac{1}{2} v$ & $0$ & $0$ & $\frac{1}{2} v$ & $0$\\
\end{tabular}
\end{center}
$J_{17}$ is a bicrossed product between the Jordan algebras $A\, =\, <a,\, b,\,c ~|~ a^{2} = a,\, b^{2} = 0,\, c^{2} = 0,\, ab=b,\, ac = \frac{1}{2} c,\, bc= 0>$ and $V = <u,\, v ~|~ u^{2} = u,\, v^{2} = 0,\, uv = \frac{1}{2} v>$. The actions $(\triangleleft, \, \triangleright)$ of the associated matched pair are defined as follows:
$$
u \triangleright c =  \frac{1}{2} c,\,\,\,\, v \triangleleft a = \frac{1}{2} v.
$$
\end{examples}

The bicrossed product of two Jordan algebras is the construction
responsible for the \emph{factorization problem} for Jordan
algebras as formulated in the introduction. Indeed, we can prove the Jordan algebra version of
\cite[Theorem 3.9]{LW}:

\begin{theorem}\thlabel{bicromlver}
A Jordan algebra $E$ factorizes through two given Jordan algebras
$A$ and $V$ if and only if there exists a matched pair of Jordan
algebras $(A, V, \triangleleft, \, \triangleright)$ such that $E
\cong A \bowtie V $.
\end{theorem}

\begin{proof}
First observe that $A \cong A\times \{0\}$ and $V \cong
\{0\}\times V$ are Jordan subalgebras of $A \bowtie V $ and obviously $A \bowtie V $ factorizes through $A\times \{0\}$ and
$\{0\}\times V$.

Conversely, assume that a Jordan algebra $E$ factorizes through two Jordan subalgebras $A$ and $V$. Consider $\pi_{A} \colon E \to A$ be the canonical projection, i.e. $\pi_{A}(a) = a$ for all $a \in A$ and ${\rm ker}\, \pi_{A} = V$. We can now define the following:
\begin{eqnarray}
\triangleright \colon V \times A \to
A, \qquad x \triangleright a &:=& \pi_{A} (x a ) = \pi_{A}(a x ) \eqlabel{bala1}\\
\triangleleft \colon V \times A \to V,
\qquad x \triangleleft a &:=& x a - \pi_{A} (x a) \eqlabel{bala2}
\end{eqnarray}
We will show that $\bigl(A,\, V,\, \triangleleft, \,
\triangleright \bigl)$ is a matched pair of Jordan algebras and $ \varphi: A \bowtie V \to
E$, $\varphi(a, x) := a + x$ is an isomorphism of Jordan algebras. To this end note that $\varphi$ is a linear isomorphism between $E$ and the direct product of vector spaces $A \times V$ with the inverse given by $\varphi^{-1}(y) := \bigl(\pi_{A}(y), \, y -
\pi_{A}(y)\bigl)$, for all $y \in E$. Therefore, there exists a unique Jordan
algebra structure on $A \times V$ such that $\varphi$ becomes an
isomorphism of Jordan algebras and this unique multiplication on
$A \times V$ is given for any $a$, $b \in A$ and $x$, $y\in V$ by:
$$
(a, x) \circ (b, y) := \varphi^{-1} \bigl( \varphi(a, x) \,
 \varphi(b, y) \bigl)
$$
The following straightforward computation shows that the above multiplication indeed coincides with the one associated to the matched pair $\bigl(A,\, V,\, \triangleleft, \,
\triangleright \bigl)$ as defined in \equref{bracbicross}:
\begin{eqnarray*}
(a, x) \circ (b, y) &=& \varphi^{-1} \bigl(\varphi(a, x) \,  \varphi(b, y)\bigl) = \varphi^{-1} \bigl( (a+x) (b+y) \bigl) \\
&=& \varphi^{-1} (a b + a y + x b + x
y) \\
&=& \bigl( \pi_{A}(a b) + \pi_{A}(a y ) + \pi_{A} (x b) +
\pi_{A}(x y), \, \\
&& a b + a y + x b + x y - \pi_{A}(a b) - \pi_{A}(a y )- \pi_{A} (x b) -
\pi_{A}(x y)  \bigl)\\
&=& \Bigl(a b + x \triangleright b + y \triangleright a, \,\,  x \triangleleft b + y \triangleleft a + x\, y \Bigl)
\end{eqnarray*}
\end{proof}

Based on \thref{bicromlver} we can restate the factorization
problem as follows:

\emph{Let $A$ and $V$ be two given Jordan algebras. Describe the
set of all matched pairs $(A, V, \, \triangleleft, \,
\triangleright)$ and classify up to an isomorphism all bicrossed
products $A \bowtie V$}.

In this way, the factorization problem can be divided into two
parts: the first one is a computational one which requires the
explicit description of all matched pairs $(\triangleleft, \,
\triangleright)$ between Jordan algebras $A$ and $V$ while the
second one deals with the classification of all bicrossed products
$A \bowtie V$ associated to the matched pairs previously
described. There is no general strategy for approaching the first
part of the problem which needs to be treated 'case by case' for
all specific examples of Jordan algebras (in the spirit of
\cite{a, abm1}). The tool for approaching the second part of the
problem is the next result which can be also used to describe the
automorphisms group ${\rm Aut}_J ({A \bowtie V})$ of a given
bicrossed product. In light of \exref{mpbpex}, the classification of bicrossed products is an important step in the classification of all Jordan algebras of a given dimension.

\begin{theorem}\thlabel{morfismebicross}
Let $(A, V, \triangleleft, \triangleright)$ and $(A', V',
\triangleleft', \triangleright')$ be two matched pairs of Jordan
algebras. Then there exists a bijective correspondence between the
set of all morphisms of Jordan algebras $\psi : A \bowtie V \to A'
\bowtie ' V' $ and the set of all quadruples $(r, s, t, q)$, where
$r: A \to A'$, $s: A \to V'$, $t: V \rightarrow A'$, $q: V
\rightarrow V'$ are linear maps satisfying the following
compatibility conditions for any $a$, $b\in A$ and $x$, $y\in V$:
\begin{eqnarray}
r(a\, b) - r(a) \, r(b) &{=}&  s(a) \triangleright' r(b) + s(b) \triangleright' r(a) \eqlabel{C1}\\
s(a \ b) - s(a) \, s(b) &{=}& s(a) \triangleleft' r(b) + s(b) \triangleleft' r(a)  \eqlabel{C2}\\
t(x \, y) - t(x) \, t(y)  &{=}&  q(x) \triangleright' t(y) + q(y) \triangleright' t(x)   \eqlabel{C3}\\
q(x\, y) - q(x) \, q(y)  &{=}&  q(x)\triangleleft' t(y) + q(y) \triangleleft' t(x)  \eqlabel{C4}\\
r(x \triangleright a) + t (x \triangleleft a) &{=}& r(a) \, t(x) + s(a)  \triangleright' t(x) + q(x)  \triangleright' r(a) \eqlabel{C5}\\
s(x \triangleright a) + q (x \triangleleft a)  &{=}& s(a) \, q(x) + s(a)  \triangleleft' t(x) + q(x)
\triangleleft' r(a) \eqlabel{C6}
\end{eqnarray}
Under the above bijection the Jordan algebra map $ \psi =
\psi_{(r, s, t, q)} : A \bowtie V \to A' \bowtie ' V' $
corresponding to $(r, s, t, q)$ is given by:
\begin{equation}\eqlabel{morfbicros}
\psi \bigl((a, \, x)\bigl) = \bigl(r(a) + t(x) , \, s(a) + q(x)
)\bigl)
\end{equation}
for all $a \in A$ and $x\in V$.
\end{theorem}

\begin{proof}
By the universal property of the direct product of vector spaces
we obtain that for any linear map $\psi : A \bowtie V \to A'
\bowtie ' V'$ there exists a unique quadruple $(r, s, t, q)$ of
linear maps $r: A \to A'$, $s: A \to V'$, $t: V \rightarrow A'$,
$q: V \rightarrow V'$ such that $\psi \bigl((a, \, x)\bigl) =
\bigl(r(a) + t(x) , \, s(a) + q(x) )\bigl)$, for all $a \in A$ and
$x\in V$. It remains to investigate when such a map $\psi =
\psi_{(r, s, t, q)}$ is a morphism of Jordan algebras between the
two bicrossed products, i.e. the following holds for all $(a, x),
(b, y) \in A\times V$:
\begin{equation}\eqlabel{morftest}
\psi \bigl( (a, \, x) \circ_{A \bowtie V} \, (b, \, y)  \bigl) =
\psi (a, x) \circ_{A' \bowtie ' V'} \, \psi (b, y)
\end{equation}
This is again a rather cumbersome but straightforward computation
which will be skipped. We only indicate the main strategy for the
proof: first, we can prove that \equref{morftest} holds for the
pairs $(a, 0)$ and $(b, 0)$ if and only if \equref{C1} and
\equref{C2} are fulfilled. Secondly, it can be shown that
\equref{morftest} holds for the pairs $(0, x)$ and $(0, y)$ if and
only if \equref{C3} and \equref{C4} hold. Finally,
\equref{morftest} holds for the pairs $(a, 0)$ and $(0, x)$ if and
only if \equref{C5} and \equref{C6} hold and this finishes the
proof.
\end{proof}

\begin{example}\exlabel{mpabelian}
Let $A_{0}$ be an abelian Jordan algebra and consider $k_0$ to be the abelian Jordan algebra of dimension $1$. Then 
there exists a bijection between the set of all matched pairs of Jordan algebras $(A_0, k_0, \triangleleft, \triangleright)$ and the set of 
all linear endomorphisms $D \in {\rm End}_k (A)$ such that $D^3 = 0$ and the matched pair $(\triangleleft, \triangleright)$ 
associated to $D$ is given as follows for all $\alpha \in k$ and $a\in A$: 
\begin{equation} \eqlabel{mpab1}
\alpha \triangleleft a := 0, \qquad \alpha \triangleright a := \alpha D(a)
\end{equation}
In particular, any bicrossed product $A_0 \bowtie k_0$ of Jordan algebras is a 
left semidirect product $A_0 \ltimes k_0$, which we denote by $A(D) := A \times k$ and whose multiplication is given 
for all $a$, $b\in A$ and $\alpha$, $\beta \in k$ by: 
$$
(a, \, \alpha) \circ (b, \, \beta) = (\alpha \, D(b) + \beta \, D(a), \, 0)
$$
Indeed, since $V = k$, there exists a bijection between the set of all bilinear maps $(\triangleleft, \triangleright)$, 
$\triangleleft : k \times A \to k$ and $\triangleright : k\times A \to A$ and the set of all linear maps 
$(\lambda, \, D) \in A^* \times {\rm End}_k (A)$ given such that 
$$
\alpha \triangleleft a := \alpha \lambda (a), \qquad \alpha \triangleright a := \alpha D(a)
$$
for all $\alpha \in k$, $a\in A$. Now since both $A$ and $k_0$ are abelian Jordan algebras, a laborious but straightforward computation shows that the axioms (MP1)$-$(MP6) from \deref{mpmL} are reduced to the following two conditions: $\lambda = 0$ and $D^3 (a) = 0$, for all $a\in A$. 
\end{example}

\section{Application: classifying complements for Jordan algebras}

Let $A \subseteq E$ be a Jordan subalgebra of $E$. A Jordan
subalgebra $B$ of $E$ is called a \emph{complement} of $A$ in $E$
(or an \emph{$A$-complement} of $E$) if $E$ factorizes through $A$
and $B$, i.e. $E = A + B$ and $A \cap B = \{0\}$.
\thref{bicromlver} implies that the Jordan algebra $B \cong
\{0\}\times B$ is a complement of $A \cong A\times \{0\}$ in the
bicrossed product $A \bowtie B$. Moreover, if $B$ is an 
$A$-complement of $E$, then there exists a matched pair of Jordan
algebras $(A, B, \triangleright, \triangleleft)$ such that the
corresponding bicrossed product $A \bowtie B$ is isomorphic as a
Jordan algebra with $E$. The actions of the matched pair $(A, B,
\triangleright, \triangleleft)$ arise from the unique
decomposition:
\begin{equation}\eqlabel{Lie456}
(0,\,b) \circ (a,\,0) = (b \triangleright a,\, b \triangleleft a)
\end{equation}
for all $a \in A$, $b \in B$. The matched pair constructed in
\equref{Lie456} will be called the \emph{canonical matched pair}
associated to the $A$-complement $B$ of $E$.

For a Jordan subalgebra $A$ of $E$ we denote by ${\mathcal F} (A,
\, E)$ the isomorphism classes of $A$-complements of $E$. The
\emph{factorization index} of $A$ in $E$ is defined as $[E : A]^f
:= |\, {\mathcal F} (A, \, E) \,|$.

\begin{definition} \delabel{deformaplie}
Let $(A, B, \triangleright, \triangleleft)$ be a matched pair of
Jordan algebras. A $k$-linear map $r: B \to A$ is called a
\emph{deformation map} of the matched pair $(A, B, \triangleright,
\triangleleft)$ if the following compatibility holds for all $x$,
$y \in B$:
\begin{equation}\eqlabel{factLie}
r(x y) \, - \, r(x) r(y) = x \triangleright
r(y) + y \triangleright r(x) - r \bigl( x \triangleleft r(y) + y
\triangleleft r(x) \bigl)
\end{equation}
We denote by ${\mathcal D}{\mathcal M} \, (B, A\, | \,
(\triangleright, \triangleleft) )$ the set of all deformation maps
of the matched pair $(A, B, \triangleright, \triangleleft)$. 
\end{definition}

\begin{example}\exlabel{defmap}
Consider $J$ to be the $4$-dimensional real Jordan algebra with multiplication table defined as follows:
\begin{center}
\begin{tabular} {l | c  c  c  c  c}
$J$ & $a$ & $b$ & $u$ & $v$\\
\hline $a$ & $a$ & $0$ & $0$ & $\frac{1}{2}v$\\
$b$ & $0$ & $b$ & $0$ & $\frac{1}{2}v$ \\
$u$ & $0$ & $0$ & $u$ & $0$\\
$v$ & $\frac{1}{2}v$ & $\frac{1}{2}v$ & $0$ & $0$ \\
\end{tabular}
\end{center}
$J$ is a semidirect product between the $2$-dimensional Jordan algebras $A \,=\, <a,\, b ~|~ a^{2} = a,\, b^{2} = b, \, ab =0>$ and $V = <u,\, v ~|~ u^{2} = u,\, v^{2} = 0,\, uv = 0>$. Furthermore, the associated matched pair $(\triangleleft, \, \triangleright)$ consists of the trivial left action $\triangleright \colon V \times A \to A$ and the right action $\triangleleft \colon V \times A \to V$ defined as follows:
$$
v \triangleleft a = v \triangleleft b = \frac{1}{2} v.
$$
It can be easily seen by a straightforward computation that the deformation maps $r \colon V \to A$ associated to this matched pair are given as follows for some $\alpha \in \RR$:
\begin{eqnarray}
&& r(u) = 0,\,\,\, r(v) = \alpha\, b \eqlabel{def1}\\
&& r(u) = 0,\,\,\, r(v) = \alpha\, a \eqlabel{def2}\\
&& r(u) = a+b,\,\,\, r(v) = \alpha\, b \eqlabel{def3}\\
&& r(u) = a+b,\,\,\, r(v) = \alpha\, a \eqlabel{def4}\\
&& r(u) = a,\,\,\, r(v) = 0 \eqlabel{def5}\\
&& r(u) = b,\,\,\, r(v) = 0 \eqlabel{def6}
\end{eqnarray}
\end{example}

Using this concept the following deformation of a Jordan algebra is
proposed:

\begin{theorem}\thlabel{deforJ}
Let $A$ be a Jordan subalgebra of $E$, $B$ a given $A$-complement
of $E$ and $r\colon B \to A$ a deformation map of the associated
canonical matched pair $(A, B, \triangleright, \triangleleft)$.

$(1)$ Let $f_{r}: B \to E = A \bowtie B$ be the $k$-linear map
defined for all $x \in B$ by:
$$f_{r}(x) = \bigl(r(x),\, x\bigl)$$
Then $\widetilde{B} : = {\rm Im}(f_{r})$ is an $A$-complement of
$E$.

$(2)$ $B_{r} := B$, as a $k$-vector space, with the new multiplication defined for all $x$, $y \in B$ by:
\begin{equation}\eqlabel{rLiedef}
x\, \cdot_{r} \, y := x y + x \triangleleft r(y) + y
\triangleleft r(x)
\end{equation}
is a Jordan algebra called the \emph{$r$-deformation} of $B$.
Furthermore, $B_{r} \cong \widetilde{B}$, as Jordan algebras.
\end{theorem}

\begin{proof}
$(1)$ To start with, we will prove that $\widetilde{B} =
\{\bigl(r(x),\, x \bigl) ~|~ x \in B\}$ is a Jordan subalgebra of
$E = A \bowtie B$. Indeed, for all $x$, $y \in B$ we have:
\begin{eqnarray*}
\bigl(r(x),\, x\bigl) \circ \bigl(r(y),\, y\bigl)
&\stackrel{\equref{bracbicross}}{=}& \bigl(\underline{r(x) r(y) + x \triangleright r(y) + y \triangleright
r(x)},\,\,x \triangleleft r(y) + y \triangleleft
r(x) + x y\bigl)\\
&\stackrel{\equref{factLie}}{=}& \Bigl( r\bigl(x \triangleleft
r(y) + y \triangleleft r(x) + x y\bigl),\, x
\triangleleft r(y) + y \triangleleft r(x) + x y\Bigl)
\end{eqnarray*}
and the latter term obviously belongs to $\widetilde{B}$. Moreover, it is straightforward to see that $A \cap \widetilde{B}
= \{0\}$ and $(a,\, b) = \bigl(a - r(b),\,0\bigl)+\bigl(r(b),\,b
\bigl) \in A + \widetilde{B}$. Therefore, $\widetilde{B}$ is a
$A$-complement of $E$.

$(2)$ We denote by $\widetilde{f_{r}}$ the $k$-linear isomorphism
from $B$ to $\widetilde{B}$ induced by $f_{r}$. We will prove that
$\widetilde{f_{r}}$ is also a Jordan algebra map if we consider on
$B=B_r$ the multiplication given by \equref{rLiedef}. Indeed, for any
$x$, $y \in B$ we have:
\begin{eqnarray*}
\widetilde{f_{r}}\bigl(x\, \cdot_{r} \,
y\bigl)&\stackrel{\equref{rLiedef}}{=}& \widetilde{f_{r}}\bigl(x y + x \triangleleft r(y) + y \triangleleft
r(x)\bigl)\\
&{=}& \bigl(\underline{r(x y + x \triangleleft r(y) + y
\triangleleft r(x))},\, x y + x \triangleleft r(y) + y
\triangleleft
r(x)\bigl)\\
&\stackrel{\equref{factLie}}{=}& \bigl(r(x) r(y) + x
\triangleright r(y) + y \triangleright r(x),\, x y + x
\triangleleft r(y) + y \triangleleft
r(x)\bigl)\\
&\stackrel{\equref{bracbicross}}{=}& \bigl(r(x),\,x\bigl) \circ
\bigl( r(y),\, y\bigl) = \widetilde{f_{r}}(x)\, \circ
\widetilde{f_{r}}(y)
\end{eqnarray*}
This shows that $B_{r}$ is a Jordan algebra and the proof is
now finished.
\end{proof}

We are now able to describe all complements of a Jordan subalgebra
$A$ of $E$.

\begin{theorem} \thlabel{descrierecomlie}
Let $A$ be a Jordan subalgebra of $E$, $B$ a given $A$-complement
of $E$ with the associated canonical matched pair of Jordan
algebras $(A, B, \triangleright, \triangleleft)$. Then
$\overline{B}$ is an $A$-complement of $E$ if and only if there
exists an isomorphism of Jordan algebras $\overline{B} \cong
B_{r}$, for some deformation map $r \colon B \to A$ of the matched
pair $(A, B, \triangleright, \triangleleft)$.
\end{theorem}

\begin{proof}
Let $\overline{B}$ be an arbitrary $A$-complement of $E$. Since $E
= A \bowtie B = A \bowtie \overline{B}$ we can find four
$k$-linear maps:
$$
u: B \to A, \quad v: B \to \overline{B}, \quad t:\overline{B} \to
A, \quad w: \overline{B} \to B
$$
such that for all $x \in B$ and $y \in \overline{B}$ we have:
\begin{equation} \eqlabel{lie111}
(0,\, x) = \bigl(u(x),\, v(x)\bigl), \qquad (0,\,y) =
\bigl(t(y),\, w(y)\bigl)
\end{equation}
It can be easily proved that $v$ and $w$ are inverses to each
other and, in particular, that $v: B \to \overline{B}$ is a
$k$-linear isomorphism. We denote by $\tilde{v}: B \to A \bowtie
B$ the composition:
$$
\tilde{v} : \, B \, \stackrel{v} {\longrightarrow} \, \overline{B}
\, \stackrel{i}{\hookrightarrow} \, E \, = \,A \bowtie B
$$
More precisely, we have $(0,\,\tilde{v}(x)) = (0,\,v(x))
\stackrel{\equref{lie111}}{=} (-u(x),\, x)$, for all $x \in B$. Consider now $r := -u$; we will prove that $r$ is a deformation map
and $\overline{B} \cong B_{r}$. Indeed, $\overline{B} = {\rm Im}
(v) = {\rm Im} (\tilde{v})$ is a Jordan subalgebra of $E = A
\bowtie B$ and therefore we have:
\begin{eqnarray*}
\bigl(r(x),\, x\bigl) \circ \bigl(r(y),\, y\bigl)
&\stackrel{\equref{bracbicross}}{=}& \bigl(r(x) r(y) + x
\triangleright r(y) + y \triangleright r(x),\,\,x \triangleleft
r(y) + y \triangleleft r(x) + x y\bigl)\\
&=& (r(z),\, z)
\end{eqnarray*}
for some $z \in B$. Thus, we obtain:
\begin{equation}\eqlabel{lie113}
r(z) = r(x)\, r(y) + x \triangleright r(y) + y
\triangleright r(x), \qquad z = x \triangleleft r(y) + y
\triangleleft r(x) + x y
\end{equation}

Now by applying $r$ to the second part of \equref{lie113} we
obtain:
\begin{eqnarray*}
r(z) = r\bigl(x \triangleleft r(y) + y \triangleleft r(x) + x y\bigl)
\end{eqnarray*}
which combined with the first part of \equref{lie113} shows that
$r$ is a deformation map of the matched pair $(A, B,
\triangleright, \triangleleft)$. Moreover, for all $x$, $y \in B$
we have:
\begin{eqnarray*}
&& \bigl(0,\, v(x) v(y)\bigl) = \bigl(0,\, v(x)\bigl)
\circ \bigl(0,\,v(y)\bigl) = \bigl(r(x),\, x\bigl)\circ
\bigl(r(y),\,
y\bigl)= \bigl(r(z),\, z\bigl)\\
&=& \bigl(0,\, v(z)\bigl) = \bigl(0,\, v(x \triangleleft r(y) + y
\triangleleft r(x) + x y)\bigl) = \bigl(0,\, v(x
\cdot_{r} y)\bigl)
\end{eqnarray*}
that is, $v: B_{r} \to \overline{B}$ is a Jordan algebra map and
the proof is now finished.
\end{proof}

In order to classify all complements we need to introduce the
following:

\begin{definition}\delabel{equivLie}
Let $(A, B, \triangleright, \triangleleft)$ be a matched pair of
Jordan algebras. Two deformation maps $r$, $s: B \to A$ are called
\emph{equivalent} and we denote this by $r \sim s$ if there exists
$\sigma: B \to B$ a $k$-linear automorphism of $B$ such that for
all $x$, $y\in B$ we have:
\begin{equation}\eqlabel{equivLiemaps}
\sigma (x y) - \sigma(x)\, \sigma(y) = \sigma(x)
\triangleleft s \bigl(\sigma(y)\bigl) - \sigma\bigl(x
\triangleleft r(y)\bigl) + \sigma(y) \triangleleft s
\bigl(\sigma(x)\bigl) - \sigma \bigl(y \triangleleft r(x)\bigl)
\end{equation}
\end{definition}

We conclude this section with the following classification result
for complements of Jordan algebras:

\begin{theorem}\thlabel{clasformelorLie}
Let $A$ be a Jordan subalgebra of $E$, $B$ an $A$-complement of $E$
and $(A, B, \triangleright, \triangleleft)$ the associated
canonical matched pair. Then:

$(1)$ $\sim$ is an equivalence relation on ${\mathcal D}{\mathcal
M} \, (B, A \, | \, (\triangleright, \triangleleft) )$. We denote
by  ${\mathcal H}{\mathcal A}^{2} (B, A \, | \, (\triangleright,
\triangleleft) )$ the quotient set ${\mathcal D}{\mathcal M} \,
(B, A \, | \, (\triangleright, \triangleleft) )/ \sim$.

$(2)$ There exists a bijection between the isomorphism classes of
all $A$-complements of $E$ and ${\mathcal H}{\mathcal A}^{2} (B, A
\, | \, (\triangleright, \triangleleft) )$. In particular, the
factorization index of $A$ in $E$ is computed by the formula:
$$
[E : A]^f = | {\mathcal H}{\mathcal A}^{2} (B, A \, | \,
(\triangleright, \triangleleft) )|
$$
where $|X|$ denotes the cardinal of the set $X$.
\end{theorem}

\begin{proof}
In light of \thref{descrierecomlie}, in order to classify all
$A$-complements of $E$ it suffices to consider only
$r$-deformations of $B$, for various deformation maps $r : B \to
A$. Now let $r$, $s : B \to A$ be two deformation maps. Since $B_r =
B_s : = B$ as $k$-spaces, it follows that the Jordan algebras $B_r$
and $B_s$ are isomorphic if and only if there exists a $k$-linear
isomorphism $\sigma : B_{r} \to B_{s}$ which is also a Jordan
algebra map. It can be easily seen using \equref{rLiedef} that
$\sigma$ is a Jordan algebra map if and only if the compatibility
condition \equref{equivLiemaps} holds, i.e. $r \sim s$. Hence, $r
\sim s$ if and only if $\sigma: B_r \to B_s$ is an isomorphism of
Jordan algebras. To conclude, $\sim$ is an equivalence relation on
${\mathcal D}{\mathcal M} (B, A \, | \, (\triangleright,
\triangleleft) )$ and the map
$$
{\mathcal H}{\mathcal A}^{2} (B,\,A \, | \, (\triangleright,
\triangleleft) ) \to {\mathcal F} (A, E),  \qquad \overline{r}
\mapsto B_{r}
$$
is a bijection between sets, where $\overline{r}$ is the
equivalence class of $r$ via the relation $\sim$.
\end{proof}

We end this section with an example which computes the factorization index of a certain Jordan algebra extension.

\begin{example}\exlabel{compl}
Consider again $(A,\,V,\, \triangleright, \triangleleft)$ to be the matched pair depicted in \exref{defmap}. We will prove that the factorization index $[J : A]^f = 4$, where $J = A \bowtie V$. We start by describing the Jordan algebras $V_{r}$ corresponding to the deformation maps listed in \exref{defmap}. First note that the Jordan algebras $V_{r}$ corresponding to the $r$-deformations given in \equref{def1} and \equref{def2} are equal and defined by $V_{r} \,=\, \{u,\,v ~|~ u^{2} = u, \, v^{2} = \alpha v,\, uv=0\}$ for some $\alpha \in \RR$. If $\alpha = 0$ we obtain the Jordan algebra $V$ while if $\alpha \neq 0$ we have an isomorphism between $V_{r}$ and the Jordan algebra $V_{1} = \{u,\, v ~|~ u^{2} = u,\, v^{2} = 0,\, uv=0\}$. Furthermore, it can be easily seen that $V$ is not isomorphic to $V_{1}$. Next, the Jordan algebras $V_{r}$ corresponding to the $r$-deformations given in \equref{def3} and \equref{def4} are equal and defined by $V_{r} \,=\, \{u,\,v ~|~ u^{2} = u, \, v^{2} = \beta v,\, uv=v\}$ for some $\beta \in \RR$. If $\beta = 0$ we obtain the Jordan algebra $V_{2} = \{u,\, v ~|~ u^{2} = u,\, v^{2} = 0,\, uv=v\}$ while if $\beta \neq 0$ we have an isomorphism between $V_{r}$ and the Jordan algebra $V_{2}' = \{u,\, v ~|~ u^{2} = u,\, v^{2} = v,\, uv=v\}$. It can be easily seen that $V_{2}'$ is isomorphic to $V_{1}$ and that $V_{2}$ is not isomorphic to $V$ nor to $V_{1}$. Similarly,  the Jordan algebras $V_{r}$ corresponding to the $r$-deformations given in \equref{def5} and \equref{def6} are equal and defined by $V_{3} = \{u,\, v ~|~ u^{2} = u,\, v^{2} = 0,\, uv= \frac{1}{2}v\}$. As $V_{3}$ is not isomorphic to any of the Jordan algebras $V$, $V_{1}$ or $V_{2}$ we can conclude that $[J : A]^f = 4$, as desired.
\end{example}


\begin{thebibliography}{99}

\bibitem{a2}
Agore, A.L. - Classifying complements for associative algebras, {\sl Linear Algebra Appl.} {\bf 446} (2014), 345--355.

\bibitem{a}
Agore, A.L. - Classifying bicrossed products of two Taft
algebras, {\sl J. Pure Appl. Algebra} {\bf 222} (2018), 914--930.

\bibitem{abm1}
Agore, A.L., Bontea, C.G., Militaru, G. - Classifying bicrossed
products of Hopf algebras, {\sl Algebr. Represent. Theory} {\bf
17} (2014), 227--264.

\bibitem{acim}
Agore, A.L.,  Chirv\u asitu, A., Ion, B., Militaru, G., -
Bicrossed products for finite groups. {\sl Algebr. Represent.
Theory} {\bf 12} (2009), 481--488.

\bibitem{am-2013}
Agore, A.L., Militaru, G. - Classifying complements for Hopf
algebras and Lie algebras, {\sl J. Algebra}, {\bf 391} (2013),
193--208.

\bibitem{am-2013b}
Agore, A.L., Militaru, G. - Unified products for Leibniz algebras.
Applications, {\sl Linear Algebra Appl.} {\bf 439} (2013),
2609--2633

\bibitem{am-2015}
Agore, A.L., Militaru, G. - Classifying complements for groups. Applications,
{\sl Ann. Inst. Fourier},  {\bf 65} (2015), 1349--1365.

\bibitem{an}
Agore, A.L., N\u ast\u asescu, L.. - Bicrossed products with the Taft algebra, {\sl Arch. Math.} {\bf 113} (2019), 21--36.

\bibitem{am2022}
Agore, A.L., Militaru, G. - Unified products for Jordan algebras. Applications, arXiv:2107.04970, submitted for publication.

\bibitem{anc}
Ancochea Bermudez, J.M. , Campoamor-Stursberg, R., Garcia Vergnolle, L., Sanchez Hernandez, J. - Contractions de algebres de Jordan en dimension 2,  {\sl J. Algebra} {\bf 319} (2008),
2395--2409.

\bibitem{Tom}
Brzezi\'nski, T. - Deformation of algebra factorisations, {\sl Comm. Algebra} {\bf 29} (2001), 737--748.

\bibitem{CIMZ}
Caenepeel, S.,  Ion, B.,  Militaru, G.  and  Zhu, S. - The
factorization problem and the smash biproduct of algebras and
coalgebras. {\sl Algebr. Represent. Theory} {\bf 3} (2000), 19--42.

\bibitem{cap}
Cap, A., Schichl, H., Vanzura, J. - On twisted tensor product of
algebras, {\sl Comm. Algebra} {\bf 23} (1995), 4701 -- 4735. 

\bibitem{CDDV}
Carotenuto, A., Dabrowski, L. and Dubois-Violette, M. -
Differential calculus on Jordan algebra and Jordan modules, {\sl
Lett. Math. Phys.},  {\bf 109} (2019), 113--133.

\bibitem{Gk}
Gelaki, S. - Exact factorizations and extensions of fusion categories, {\sl J. Algebra} {\bf 480} (2017), 505--518.

\bibitem{jacb}
Jacobson, N. -  Structure and Representations of Jordan Algebras,
Amer. Math. Soc., 1968.

\bibitem{Ka}
Kashuba, I. - Variety of Jordan algebras in small dimension, {\sl Algebra Discrete Math.}, {\bf 2}(2006),
62--76.

\bibitem{KOS}
Kashuba, I., Ovsienko, O., Shestakov, I. - Representation type
of Jordan algebras, {\sl Adv. Mathematics}, {\bf 226}(2011),
385--418.

\bibitem{K}
Keilberg, M. - Automorphisms of the doubles of purely non-abelian finite groups,
{\sl Algebr. Represent. Theory} {\bf 18} (2015), 1267--1297.

\bibitem{Ko}
Koecher, M. - The Minnesota Notes on Jordan Algebras and their
Applications, Lecture Notes in Math. 1710, Springer Verlag,
Berlin, 1999.

\bibitem{Y}
Kosmann-Schwarzbach, Y., Magri, F. - Poisson-Lie groups and complete integrability. I. Drinfel'd bialgebras, dual extensions and their canonical representations, {\sl Ann. Inst. H. Poincare Phys. Theor.} {\bf 49} (1988), 433--460.

\bibitem{LW}
Lu, J.H. and Weinstein, A. - Poisson Lie groups, dressing
transformations and Bruhat decompositions,  {\sl J. Differential
Geom.},  {\bf 31}(1990), 501--526.

\bibitem{Ma}
E. Maillet, Sur les groupes \' echangeables et les groupes d\'
ecomposables. {\sl Bull. Soc. Math. France}  {\bf 28}  (1900),
7--16.

\bibitem{majid3}
Majid, S. - Matched pairs of Lie groups and Hopf algebra
bicrossproducts, {\sl Nuclear Physics B} {\bf 6} (1989), 422--424.

\bibitem{majid}
Majid, S. - Physics for algebraists: non-commutative and
non-cocommutative  Hopf algebras by a bicrossproduct construction,
{\sl J. Algebra}, {\bf 130} (1990), 17--64.

\bibitem{Mc}
McCrimmon, K. - A taste of Jordan algebras, Universitext,
Springer-Verlag, 2004. 

\bibitem{Or} O. Ore, Structures and group theory. I., {\sl Duke Math. J.}
{\bf 3} (1937), no. 2, 149--174.

\bibitem{Tak}
Takeuchi, M. - Matched pairs of groups and bismash products of
Hopf algebras, {\sl Comm.  Algebra} {\bf 9}(1981), 841--882.

\end{thebibliography}
\end{document}